\documentclass[12pt]{article}
\usepackage[utf8]{inputenc}
\usepackage{amssymb, amsmath,amsthm, mathtools}
\usepackage{bbm}
\usepackage{amsfonts}
\usepackage{xcolor}

\usepackage{hyperref}

\usepackage{tabto}
\TabPositions{2 cm, 4 cm, 6 cm, 8 cm}

\usepackage{tikz-cd}

\usepackage{tikz}
\usetikzlibrary{arrows, automata, positioning}

\usepackage{tocloft}
\usepackage{appendix}

\usepackage{enumerate}

\usepackage{hyperref}

\hypersetup{
  colorlinks   = true, 
  urlcolor     = blue, 
  linkcolor    = blue, 
  citecolor   = blue 
}

\newtheorem{theorem}{Theorem}[section]
\newtheorem{corollary}[theorem]{Corollary}
\newtheorem{lemma}[theorem]{Lemma}

\newtheorem{proposition}[theorem]{Proposition}

\theoremstyle{definition}
\newtheorem{definition}[theorem]{Definition}
\newtheorem{example}[theorem]{Example}
\newtheorem{remark}[theorem]{Remark}

\newtheorem{question}[theorem]{Question}

\newtheorem*{notation*}{Notation}
\newtheorem*{acks*}{Acknowledgements}

\renewcommand\leq{\leqslant}
\renewcommand\geq{\geqslant}

\newcommand{\Aut}{\operatorname{Aut}}
\newcommand{\Out}{\operatorname{Out}}
\newcommand{\Inn}{\operatorname{Inn}}
\newcommand{\ad}{\operatorname{ad}}

\newcommand{\cl}{\operatorname{cl}}

\newcommand{\normal}[1]{\left<\! \left< #1\right> \!\right>}

\newcommand{\Gn}{\mathcal{G}_n}

\begin{document}

\title{Infinite simple characteristic quotients}
\author{R{\'e}mi Coulon and Francesco Fournier-Facio}
\date{\today}
\maketitle

\begin{abstract}
We construct continuum many infinite, simple, characteristic quotients of non-abelian free groups, answering a 1978 question of James Wiegold. The method is very flexible, allowing to impose certain properties on the quotients, to generalize the construction to large classes of groups with hyperbolic features, and to produce quasi-isometrically diverse examples.
\end{abstract}

\section{Introduction}

A subgroup of a group $\Gamma$ is called \emph{characteristic} if it is invariant under all automorphisms of $\Gamma$, and by extension the corresponding quotients are called \emph{characteristic}.
Wiegold asked if non-abelian free groups admit infinite simple characteristic quotients  \cite[Question 6.45]{kourovka}. 
The main goal of this paper is to prove the following theorem, answering positively Wiegold's question.

\begin{theorem}[Corollary \ref{cor:answer}]
\label{intro:thm:answer}
    For all $n \geq 2$, the free group of rank $n$ admits continuum many, pairwise non-isomorphic, $2$-generated, infinite, simple, characteristic quotients.
\end{theorem}

We will state stronger versions later in the introduction (see Theorems \ref{intro:thm:answer:strong} and \ref{intro:thm:QI}).

\medskip

Wiegold's question on the (non-)existence of such a group falls in the context of \emph{growth sequences} of simple groups. For a finitely generated group $\Gamma$, we denote by $d(\Gamma)$ its rank, that is the smallest $n$ such that $\Gamma$ is a quotient of the free group of rank $n$, denoted by $F_n$. The growth sequence of $\Gamma$ is the sequence $\{ d(\Gamma^p) \}_{p \geq 1}$.
Growth sequences of finite groups were studied by Wiegold and co-authors, who gave very precise asymptotic formulas: see  \cite{w:finite1, w:finite6} and the references therein.

The case of infinite groups is more mysterious, and was also studied by Wiegold and several other authors: see \cite{obratsov, erfanian} and the references therein, as well as  \cite{wise, bridson} for approaches of a different flavor. In particular, Wiegold and Wilson proved that if $\Gamma$ is a finitely generated infinite simple group, then $d(\Gamma^p) \leq d(\Gamma) + 1$ for all $p \geq 1$ \cite{w:fg1}. However, there seems to be no example of a finitely generated infinite simple group whose growth sequence is not constant. The question on the existence of such a group, and more precisely of a finitely generated infinite simple group such that $d(\Gamma^2) = d(\Gamma) + 1$, has been called ``irreducibly difficult'' \cite{w:fg1, obratsov}.

\medskip

This last question admits a reformulation purely at the level of free groups. For a group $\Gamma$ of rank at most $n$, let $\mathcal{N}_n(\Gamma) \coloneqq \{ N \triangleleft F_n : F_n/N \cong \Gamma \}$ be the set of normal subgroups of $F_n$ whose quotient is isomorphic to $\Gamma$. An observation of Wiegold \cite{w:tarski} shows that if $\Gamma$ is a finitely generated infinite simple group of rank $n$, then $d(\Gamma^2) = n+1$ if and only if $\# \mathcal{N}_n(\Gamma) = 1$. This naturally led Wiegold to ask for two weaker examples:
\begin{itemize}
    \item \cite[Question 16.100]{kourovka} If $n \geq 2$, does there exist an infinite simple group $\Gamma$ of rank at most $n$ such that the action of $\Aut(F_n)$ on $\mathcal{N}_n(\Gamma)$ is transitive?\footnote{In \cite[Question 16.100]{kourovka}, Wiegold asks this question for $n = 2$ only. For $n \geq 4$, it was answered affirmatively in \cite[Theorem 4]{garionglasner}.}
    \item \cite[Question 6.45]{kourovka} If $n \geq 2$, does there exist an infinite simple group $\Gamma$ of rank at most $n$ such that the action of $\Aut(F_n)$ on $\mathcal{N}_n(\Gamma)$ has a fixed point?
\end{itemize}
Theorem \ref{intro:thm:answer} answers the latter question in the affirmative for all $n \geq 2$.

\medskip

These questions also fall into the broader context of the study of the dynamics of $\Aut(F_n)$ on the space of \emph{marked groups} $\mathcal{G}_n$. This is a compact space that parametrizes all groups of rank at most $n$, and is a central tool to study geometric and model-theoretic properties of finitely generated groups \cite{grigorchuk, champetier, limit, diversity, osin:law}. The dynamics of $\Aut(F_n)$ on $\mathcal{G}_n$ is very rich, and can shed light on generic properties of finitely generated groups \cite{osin:aut}. The above questions arise very naturally in this context, and are instances of the following general problem: To which extent do certain algebraic properties of the underlying isomorphism type influence the orbit of $\Aut(F_n)$ on a marked group? 

\begin{remark}
\label{rem:finite}
    Analogous conjectures on finite groups have also received much attention:
    \begin{itemize}
        \item \cite[Conjecture 2.5.4]{pak:question} If $n \geq 3$, and $\Gamma$ is a finite simple group, then the action of $\Aut(F_n)$ on $\mathcal{N}_n(\Gamma)$ is transitive.
        \item \cite[p.25]{lubotzky:question} If $n \geq 3$, and $\Gamma$ is a finite simple group, then the action of $\Aut(F_n)$ on $\mathcal{N}_n(\Gamma)$ has no fixed points.
    \end{itemize}
    The first conjecture is commonly attributed to Wiegold, and the case of $n = 2$ is known to fail \cite{counterexample, counterexample:statistic}. The second conjecture is by Lubotzky, and was recently shown to fail for $n = 2$ \cite{counterexample:baby, counterexample:baby2}.
\end{remark}

Besides the ranks of products of infinite simple groups, the ranks of infinite simple groups themselves are also not well understood. There is only one known construction of a sequence of finitely generated simple groups with unbounded rank, due to Osin and Thom \cite{osinthom}. 
Our construction does not directly provide simple quotients with arbitrarily large rank.
On the contrary, as we have stated, one can easily ensure that they are $2$-generated.
In particular, this provides a negative answer to \cite[Question 17.62]{kourovka}, which asks whether every characteristic quotient of $F_n$ has rank $n$.

\medskip

Our construction of the quotients in Theorem \ref{intro:thm:answer} boils down to the following general result:

\begin{theorem}[Corollary \ref{cor:general}]
\label{intro:thm:general}
    Let $G$ be an acylindrically hyperbolic group, and let $H < G$ be finitely generated, infinite and normal. Then there exists a subgroup $N < H$, normal in $G$, such that $H/N$ is $2$-generated, infinite, and simple.
\end{theorem}

As a corollary we obtain:

\begin{corollary}[Corollary \ref{cor:aut}]
\label{intro:cor:aut}
    Let $\Gamma$ be a finitely generated group such that $[\Gamma : Z(\Gamma)] = \infty$ and $\Aut(\Gamma)$ is acylindrically hyperbolic. Then $\Gamma$ admits an infinite, $2$-generated, simple, characteristic quotient.
\end{corollary}

The existence of infinite simple characteristic quotients of $F_n$ (i.e., Wiegold's question \cite[Question 6.45]{kourovka}) then follows, by the acylindrical hyperbolicity of $\Aut(F_n)$, due to Genevois and Horbez \cite{auto:infend}. In fact, acylindrical hyperbolicity is known in much greater generality, due to several authors \cite{auto:graph2, auto:graph1, auto:oneend, auto:infend, escalierhorbez, auto:artin}, which allows us to obtain infinite simple characteristic quotients in large classes of negatively curved groups; see Example \ref{ex:list} for a comprehensive list.

\medskip

Theorem~\ref{intro:thm:general} is proven via a small cancellation construction over acylindrically hyperbolic groups. This construction is quite flexible, and we obtain additional control over the infinite simple characteristic quotients. In particular, we can map finite sets injectively, preserve the torsion, and embed outer automorphism groups. Moreover, we can embed any countable subgroup in an infinite simple characteristic quotient. This leads to the following stronger and more general version of Theorem \ref{intro:thm:answer}:

\begin{theorem}[Theorem \ref{thm:answer}]
\label{intro:thm:answer:strong}
    Let $\Gamma$ be a torsion-free, non-elementary hyperbolic group. Then, for every countable group $L$, there exists a characteristic subgroup $N < \Gamma$ such that $\Gamma / N$ is simple, $2$-generated, contains $L$ as a subgroup, and contains an element of order $n \in \mathbb{N}$ if and only if $L$ does.
\end{theorem}

It is well-known that a group $\Gamma$ as in the statement is SQ-universal, that is, every countable group embeds in a quotient of $\Gamma$ \cite{SQ1, SQ2}. Theorem \ref{intro:thm:answer:strong} shows that SQ-universality can be achieved just using simple characteristic quotients. By choosing $L$ in a family of pairwise non-isomorphic finitely generated groups, we deduce in particular that $\Gamma$ admits continuum many pairwise non-isomorphic simple characteristic quotients, and moreover the torsion can be prescribed (Corollary \ref{cor:answer}).

\medskip

We also strengthen Theorem \ref{intro:thm:answer} in another direction, by producing different examples up to quasi-isometry:

\begin{theorem}[Theorem \ref{thm:QI}]
\label{intro:thm:QI}
    There exist continuum many, pairwise non-quasi-isometric, infinite, $2$-generated, simple, characteristic quotients of $F_n$.
\end{theorem}

Theorem \ref{thm:QI} will be proven via a criterion for quasi-isometric diversity in the space of marked groups \cite{diversity}. In order to apply it, we will need to slightly modify our construction to produce infinite simple characteristic quotients of $F_n$ as marked limits of finitely presented acylindrically hyperbolic characteristic quotients of $F_n$.

\medskip

We end the paper with two questions on possible extensions and variations of our results.

\begin{acks*}
The first author acknowledges support from the Agence Nationale de la Recherche under the grant GOFR (ANR-22-CE40-0004).
The second author is supported by the Herchel Smith Postdoctoral Fellowship Fund. 
The authors are indebted to Yash Lodha for referring them to Wiegold's question; and to Anthony Genevois, Camille Horbez, Ashot Minasyan and Markus Steenbock for comments on a previous version.
The authors would also like to thank the organizers of the conference \emph{Groups and Rigidity Around the Zimmer Program} held in Ventotene in September 2023, where this work started.
Finally, the authors thank anonymous referees for their comments and suggestions.
\end{acks*}

\section{Small cancellation theory}

A method to construct simple quotients of groups goes as follows. Start with a finitely generated ``negatively curved'' group $\Gamma$, and for an element $1 \neq g \in \Gamma$ add relations that ensure that each generator of $\Gamma$ belongs to the normal closure of $g$. If this is done in a way that the quotient is still ``negatively curved'', then the process can be iterated, and the direct limit will be normally generated by each of its non-identity elements, thus simple. Versions of this construction can be found e.g. in \cite{Olshanskii:1979b}.

What makes this process rigorous is the theory of small cancellation. In our setting, in order to obtain \emph{characteristic} quotients of $F_n$, we will perform a similar construction at the level of the automorphism group $\Aut(F_n)$. Therefore the right notion of ``negative curvature'' for us is acylindrical hyperbolicity.

\subsection{Acylindrically hyperbolic groups}
Acylindricity appears in the work of Bowditch \cite{Bowditch:2008bj} as a way to extend the techniques of hyperbolic geometry beyond the scope of proper and co-compact actions.
The class of all groups with a non-elementary, acylindrical action on a hyperbolic space was later named \emph{acylindrically hyperbolic groups} by Osin \cite{osin}.
It encompasses several groups of interest in geometric group theory, such as non-elementary hyperbolic and relatively hyperbolic groups, mapping class groups and outer automorphism groups of free groups.

\medskip 
As in \cite[III.H, Definition 1.1]{BH},
we say that a metric space is \emph{$\delta$-hyperbolic} (in the sense of Gromov) if its geodesic triangles are $\delta$-slim.
A metric space is \emph{hyperbolic}, if it is $\delta$-hyperbolic for some $\delta \in\mathbb R_+$.
\emph{In this article we always assume for simplicity that $\delta > 1$.}

\begin{definition}
    Let $G$ be a group acting by isometries on a metric space $X$. We say that the action is \emph{acylindrical} if for all $\varepsilon > 0$ there exist $R > 0$ and $n \in \mathbb{N}$ such that the following holds. For all $x, y \in X$ such that $d(x, y) \geq R$, the set $\{ g \in G : d(x, gx) \leq \varepsilon, d(y, gy) \leq \varepsilon \}$ contains at most $n$ elements.
\end{definition}

Recall that an action of a group $G$ on a hyperbolic space $X$ is \emph{non-elementary} if $G$ has more than two accumulation points in the boundary $\partial X$. 
 A group $G$ is \emph{acylindrically hyperbolic} if it admits a non-elementary, acylindrical action on a hyperbolic space $X$.
Note that the action of a subgroup of $G$ on $X$ is acylindrical, and therefore every subgroup of $G$ acting non-elementarily is also acylindrically hyperbolic.

\begin{lemma}[{\cite[Lemma 7.2]{osin}}]
\label{lem:normal}
    Let $G$ be a group with a non-elementary acylindrical action on a hyperbolic space $X$. If $H < G$ is an infinite normal subgroup, then the action of $H$ on $X$ is also non-elementary. Therefore, $H$ is acylindrically hyperbolic.
\end{lemma}

Given a (possibly infinite) generating set $A \subset G$, we denote by $\Gamma(G,A)$ the Cayley graph of $G$ with respect to $A$.
For every $R \in \mathbb R_+$, we write $B_A(R)$ for the set of all the elements of $G$ corresponding to vertices of $\Gamma(G,A)$ in the ball of radius $R$  centered at the identity.
The space $X$ appearing in the definition of acylindrically hyperbolic group can always be chosen to be a Cayley graph.

\begin{theorem}[{\cite[Theorem 1.2]{osin}}]
\label{thm:cayley}
    Let $G$ be an acylindrically hyperbolic group. 
    There exists a (typically infinite) generating set $A \subset G$ such that the Cayley graph $\Gamma(G, A)$ is hyperbolic and the action of $G$ on it is non-elementary and acylindrical.
\end{theorem}

Another important property of acylindrically hyperbolic groups is that they have maximal finite normal subgroups.

\begin{theorem}[{\cite[Theorem 2.23]{DGO}, \cite[Lemma 5.10]{hull}}]
\label{thm:K}
    If $G$ is an acylindrically hyperbolic group, then $G$ contains a unique maximal finite normal subgroup $K$. 
    Moreover, $G/K$ is acylindrically hyperbolic, and has no non-trivial finite normal subgroup.
\end{theorem}

In the remainder of the article we denote by $K(G)$ the unique maximal finite normal subgroup of an acylindrically hyperbolic group $G$.

\begin{definition}
	Let $G$ be a group with a non-elementary acylindrical action on a hyperbolic space $X$.
	Let $H$ be a subgroup of $G$.
	\begin{itemize}
		\item $H$ is \emph{elliptic} (with respect to this action) if some, hence any, orbit of $H$ is bounded.
		\item $H$ is \emph{suitable} (with respect to this action) if the action of $H$ on $X$ is non-elementary (i.e., neither elliptic nor virtually cyclic) and $H$ does not normalize a non-trivial finite subgroup of $G$.
	\end{itemize}
\end{definition}

Note that a suitable subgroup can only exist if $K(G) = \{ 1 \}$.
A good source of suitable subgroups is normal subgroups \cite[Lemma 3.23]{CIOS}. Here we will need something slightly stronger: recall that a subgroup $H < G$ is \emph{subnormal} if there exists a chain of subgroups $H = H_0 < H_1 < \cdots < H_k = G$ such that $H_i$ is normal in $H_{i+1}$ for $0 \leq i < k$.

\begin{lemma}
\label{lem:normal:suitable}
	Let $G$ be a group with a non-elementary acylindrical action on a hyperbolic space $X$.
	Assume that $K(G)  = \{1\}$.
	Then every non-trivial subnormal subgroup of $G$ is suitable with respect to this action.
\end{lemma}

\begin{proof}
	Let $H$ be a non-trivial subnormal subgroup of $G$, and let $H = H_0 < \cdots < H_k = G$ be a chain witnessing this.
    The proof is by induction on the length $k$ of this chain.
    The case $k = 0$ is trivial.
    Assume now that  $k \geq 1$ and the statement holds for $k-1$.
    In particular, $H_1$ is suitable, hence non-elementary.
	It follows from Lemma~\ref{lem:normal} that the action of $H_0$ on $X$ is acylindrical and non-elementary.
    By \cite[Lemma 5.5]{hull}, there exists a unique maximal finite subgroup $K \subset G$ that is normalized by $H_0$. Let $g \in H_1$ and $h \in H_0$.
	Since $H_0$ is normal in $H_1$, the element $g^{-1}hg \in H_0$ normalizes $K$, or equivalently $h$ normalizes $gKg^{-1}$.
    By uniqueness we obtain $gKg^{-1} = K$. 
	Consequently $H_1$ normalizes $K$.
    Since $H_1$ is suitable, this forces $K = \{1\}$.
\end{proof}

\begin{lemma}
\label{lem:suitable:rank2}
    Let $G$ be a group with a hyperbolic Cayley graph $X$, such that the action of $G$ on $X$ is non-elementary and acylindrical.
    Let $H < G$ be suitable with respect to this action. Then there exists a $2$-generated subgroup $L < H$ that is also suitable with respect to this action.
\end{lemma}

\begin{proof}
    Let $g \in G$ be a loxodromic element. Then there exists a unique maximal elementary subgroup containing $g$ \cite[Lemma 6.5, Corollary 6.6]{DGO}, which we denote by $E(g)$. By \cite[Corollary 5.7]{hull}, the suitable subgroup $H$ contains two non-commensurable elements $h_1, h_2$ such that $E(h_i) = \langle h_i \rangle$ for $i = 1, 2$. In particular, $h_1 \notin E(h_2)$, and $h_1$ does not normalize a non-trivial finite subgroup. Letting $L \coloneqq \langle h_1, h_2 \rangle$ we conclude.
\end{proof}

We need one more lemma that shows that elliptic subgroups are, in a way, uniformly bounded.

\begin{lemma}
\label{lem:elliptic}
	Let $G$ be a group.
	Let $A$ be a generating set of $G$ such that $\Gamma(G,A)$ is $\delta$-hyperbolic.
	For every elliptic subgroup $E$, there is $g \in G$ such that $gEg^{-1}$ is contained in $B_A(10\delta)$.
\end{lemma}

\begin{proof}
	There is a vertex $v \in \Gamma(G, A)$ such that the orbit $Ev$ has diameter at most $10\delta$. (Such statement appears in many places in the literature with different values in place of $10\delta$. These differences depend mostly on the definition of hyperbolic space used by the authors, see e.g. \cite[III.$\Gamma$, Lemma 3.3]{BH}.  We chose on purpose a very generous formulation. Recall that we assume that $\delta > 1$, so that we always have a vertex is our set of ``almost fixed points''.)
	However the action of $G$ on the vertex set of $\Gamma(G,A)$ is transitive, hence there is an element $g \in G$ such that $gv = v_0$ where $v_0$ is the vertex representing the identity of $G$.
	Then the orbit $(g E g^{-1}) v_0$ is contained in $B_A(10 \delta)$.
\end{proof}

\subsection{Small Cancellation Theorem}

Small cancellation theory is a powerful and versatile tool to produce and study groups with prescribed properties.
It appears in many flavors in the literature \cite{Lyndon:2001wm,Olshanskii:1991wv, Gromov:2001us, Osin:2010dx, Coulon:2014fr,DGO}.
The general philosophy is roughly  the following: If a group $G$ admits a ``non-degenerate'' action on a negatively curved space $X$, then so does any quotient of $G$ obtained by adjoining relations which are ``sufficiently independent'', hence paving the way to iterated constructions.
For this paper, we rely on the work of Hull \cite{hull}.

\begin{theorem}[Small Cancellation Theorem {\cite[Theorem 7.1]{hull}}]
\label{thm:sc}
	Let $G$ be a group and $A \subset G$ a generating set of $G$.
	Assume that $\Gamma(G,A)$ is hyperbolic and the action of $G$ on this space is acylindrical and non-elementary.
    Let $H < G$ be a finitely generated subgroup that is suitable with respect to the action on $\Gamma(G,A)$. Let $B \subset G$ be a finite subset and $R \in \mathbb R_+$. Then there exists a quotient $\pi \colon G \to \bar G$ with the following properties.
    \begin{enumerate}
        \item\label{enu:sc - hyp} There is a subset $\bar A \subset \bar G$ containing $\pi(A)$ such that  $\Gamma(\bar G, \bar A)$ is hyperbolic and the action of $\bar G$ on this space is non-elementary and acylindrical.
        \item \label{enu:sc - loc inj} $\pi$ is injective on $B_A(R)$.
        \item \label{enu:sc - onto} $\pi(B) \subset \pi(H)$.
        \item \label{enu:sc - suitable} $\pi(H)$ is a $2$-generated suitable subgroup of $\bar G$ for its action on $\Gamma(\bar G, \bar A)$.
        \item \label{enu:sc - elliptic} Every element of $\bar G$ of order $n \in \mathbb{N}$ is the image of an element of $G$ of order $n$.
        \item \label{enu:sc - fp} $\ker\pi$ is the normal closure of finitely many elements in $B \cdot H$.
    \end{enumerate}
\end{theorem}

\begin{proof}
    Without loss of generality we can assume that $B$ contains a finite generating set of $H$.
    Consider a $2$-generated suitable subgroup $L$ of $H$ as given by Lemma~\ref{lem:suitable:rank2} and \cite[Theorem 7.1]{hull} with $L$ (rather than $H$).
    In particular, (\ref{enu:sc - hyp}), (\ref{enu:sc - loc inj}), and (\ref{enu:sc - elliptic}) directly follows from \cite[Theorem 7.1]{hull}.
    Hull's statement also tells us that $\pi(B) \subset \pi(L)$, hence $\pi(B) \subset \pi(H)$, which proves (\ref{enu:sc - onto}).
    On the one hand $L \subset H$.
    On the other hand we assumed that $B$ contains a generating set of $H$, so that our previous discussion yields $\pi(H) \subset \pi(L)$.
    It follows that $\pi(H) = \pi(L)$ is $2$-generated. 
    According to \cite[Theorem 7.1]{hull} it is also a suitable subgroup of $\bar G$ for its action on $\Gamma(\bar G, \bar A)$, which completes the proof of (\ref{enu:sc - suitable}).
    The last item is a consequence of the construction; see the proof of \cite[Theorem 7.1]{hull}.
\end{proof}

\begin{remark}
    Theorem \ref{thm:sc} was later refined in various directions, see e.g. \cite[Proposition 3.3]{hull:improved} or \cite[Theorem 3.27]{CIOS}. Using these stronger small cancellation theorems would allow to improve Theorem \ref{intro:thm:general} to encompass countable groups $H$ that are not necessarily finitely generated. However, we do not know of any interesting application of this more general statement.
    In addition, the assumption of finite generation allows for an easier small cancellation setup (which for instance does not require to introduce hyperbolically embedded subgroups).
    The next corollary is the only statement in the article, where assuming that $H$ is finitely generated cannot be avoided.
    Having \emph{finitely presented} quotients as below will play a crucial role in Section~\ref{s:QI}.
\end{remark}

\begin{corollary}
\label{lem:fp}
    There exists a finitely presented, $2$-generated, acylindrically hyperbolic quotient $G_0$ of $\Aut(F_n)$ such that the composition $F_n \to \Aut(F_n) \to G_0$ is surjective, and $K(G_0) = \{ 1 \}$.
\end{corollary}

\begin{proof}
    The group $\Aut(F_n)$ is acylindrically hyperbolic, with no finite normal subgroups, see \cite{auto:infend}.
    The subgroup $F_n \cong \Inn(F_n) < \Aut(F_n)$ is infinite and normal, thus suitable (Lemma \ref{lem:normal:suitable}). We fix a finite generating set $B$ of $\Aut(F_n)$. Theorem \ref{thm:sc} produces an acylindrically hyperbolic quotient $\pi : \Aut(F_n) \to G_0$ such that $\pi(F_n)$ is $2$-generated and contains $\pi(B)$ while $K(G_0) = \{1\}$.
    This shows in particular that $F_n$ surjects onto $G_0$ which is therefore $2$-generated.
    Moreover, by Theorem \ref{thm:sc} again, $\pi$ may be chosen to be the quotient of $\Aut(F_n)$ by finitely many relations. 
    Since $\Aut(F_n)$ is finitely presented \cite{nielsen}, so is $G_0$.
\end{proof}

\section{Simple characteristic quotients}

The construction of characteristic quotients relies on the following simple but key observation.

\begin{lemma}
\label{lem:char}
    Let $\Gamma$ be a group.
    Let $\pi : \Aut(\Gamma) \to Q$ be a homomorphism. 
    Then the kernel of the composition $\pi \circ \ad : \Gamma \to \Inn(\Gamma) \to Q$ is characteristic in $\Gamma$.
\end{lemma}

\begin{proof}
    Let $g \in \ker(\pi \circ \ad)$ and $\alpha \in \Aut(\Gamma)$. Then $\ad_g \in \ker(\pi)$ so $\ad_{\alpha(g)} = \alpha \ad_g \alpha^{-1} \in \ker(\pi)$ too. This shows that $\alpha(g) \in \ker(\pi \circ \ad)$ and concludes the proof.
\end{proof}

As a warm-up with provide first a short proof of the existence of infinite simple characteristic quotients of free groups.
We thank an anonymous referee for suggesting this argument.

\begin{theorem}
\label{res: referee statement}
    Let $n\geq 2$.
    There exists an infinite simple characteristic quotient of $F_n$.
\end{theorem}

\begin{proof}
    There exists an acylindrically hyperbolic quotient $\pi \colon \Aut(F_n) \to G_0$ such that $\pi$ is surjective when restricted to $F_n$ (Corollary \ref{lem:fp}). Every quotient of $G_0$ is a characteristic quotient of $F_n$ (Lemma \ref{lem:char}). Therefore, it suffices to show that $G_0$ admits an infinite simple quotient.

    Let $H$ be Higman's group \cite{higman}: this is a finitely generated group with no non-trivial finite quotients, which is moreover acylindrically hyperbolic \cite[Corollary 4.26]{MO}. 
    By \cite[Corollary 1.6]{hull} there exists an acylindrically hyperbolic group $G_1$ that is a common quotient of $G_0$ and $H$.
    Since $G_1$ is finitely generated, it follows from Zorn's lemma that it contains a maximal proper normal subgroup $N$.
    Recall that $G_1$ is a quotient of the Higman group. 
    Hence the index of $N$ in $G_1$ must be infinite.
    Thus $G_2 = G_1/N$ is an infinite simple quotient of $G_0$.
\end{proof}

\begin{remark}
    The above proof goes through Zorn's Lemma, so the group it produces is not explicit, and it is harder to control its finer properties. We now give an alternative argument relying on the small cancellation theory techniques introduced in the previous section. It naturally provides us with additional control, which we will exploit here and in the next sections.
\end{remark}

\begin{definition}
\label{def:tau}
    For a group $\Gamma$ we denote by $\tau(\Gamma)$ the set $\{ n \in \mathbb{N} : \Gamma \text{ contains an element of order } n\}$.
\end{definition}

\begin{theorem}
\label{thm:general}
	Consider a short exact sequence
	\begin{equation*}
		1 \to H \to G \to Q \to 1,
		\end{equation*}
	where $H$ is infinite and finitely generated.
	Let $A \subset G$ be a generating set of $G$.
	Assume that $\Gamma(G,A)$ is hyperbolic and the action of $G$ on this space is acylindrical and non-elementary.
	Suppose in addition that $K(G) = \{1\}$.
	For every $R \in \mathbb R_+$, there is a subgroup $N < H$, normal in $G$, with the following properties.
	\begin{itemize}
		\item $H/N$ is simple.
		\item $H/N$ is $2$-generated but not finitely presentable.
		\item The projection $G \to G/N$ is injective on $B_A(R)$.
		\item $\tau(H/N) = \tau(H)$.
		\item $Q$ embeds in $\Out(H/N)$.
	\end{itemize}
\end{theorem}

\begin{proof}	
	By assumption $\Gamma(G,A)$ is $\delta$-hyperbolic, for some $\delta \in \mathbb R_+$.
	Without loss of generality we can assume that $R \geq 10\delta$.
    We fix a finite generating set $B$ of $H$ and an enumeration $\{h_1, h_2, \ldots \}$ of $H \setminus \{ 1 \}$.
    We are going to build by induction a nested sequence of subgroups $N_k \subset H$, normal in $G$, with the following properties.
    Let $H_k \coloneqq H/N_k$ and $G_k \coloneqq G/N_k$.
    \begin{enumerate}
        \item \label{enu: induction - acyl}
        There is a generating set $A_k$ of $G_k$ containing the image of $A$ such that $\Gamma(G_k, A_k)$ is hyperbolic and the action of $G_k$ on this space is non-elementary, acylindrical.
        Moreover $K(G_k) = \{1\}$.
        \item \label{enu: induction - normal closure}
        $H_k$ is infinite. Moreover, for each $i \in \{1, \dots, k\}$, either $h_i \in N_k$, or $H_k$ is the normal closure in $H_k$ of (the image of) $h_i$.
        In addition, $H_k$ is $2$-generated, provided $k\geq 1$.
        \item \label{enu: induction - inj} 
        The projection $G \to G_k$ is injective on $B_A(R)$.
        \item \label{enu: induction - diagram}
        The following diagram commutes
        \begin{center}
             \begin{tikzcd}
             1 \arrow[r] & H  \arrow[r] \arrow[d] & G \arrow[r] \arrow[d] & Q \arrow[r]\arrow[d, "\cong"] & 1 \\
             1 \arrow[r] & H_k \arrow[r] & G_k \arrow[r] & Q \arrow[r] & 1.
            \end{tikzcd}
        \end{center}
    \end{enumerate}
    From now on, we make an abuse of notation and denote in the same way an element $g \in G$ and its images in the quotients of $G$.

    We start our induction with $N_0 = \{1\}$ for which the above assertions hold.
    Assume that we have already built the first $k+1$ subgroups $N_0< N_1 < \dots < N_k$, where $k \geq 0$. If $h_{k+1} \in N_k$, we set $N_{k+1} = N_k$ and move on to the next step.
    Let us assume that this is not the case.
    We denote by $A_k$ the generating set of $G_k$ given by the induction hypothesis~(\ref{enu: induction - acyl}).
    Recall that $A_k$ contains the image of $A$.
    Consequently the projection $G \to G_k$ induces a $1$-Lipschitz map $\Gamma(G,A) \to \Gamma(G_k, A_k)$.
    In particular, the image of $B_A(R)$ is contained in $B_{A_k}(R)$.
    Denote by $L_k$ the normal closure of $h_{k+1}$ in $H_k$.
    As $G_k$ has no non-trivial finite normal subgroup, $L_k$ is a suitable of $G_k$ with respect to its action on $\Gamma(G_k,A_k)$ (Lemma~\ref{lem:normal:suitable}).
    Denote by $B_k$ the image of $B$ in $H_k$.
    Note that $B_k$ is a finite generating set of $H_k$.
    According to the Small Cancellation Theorem (Theorem \ref{thm:sc}) there is a quotient $G_{k+1}$ of $G_k$ such that the projection $\pi \colon G_k \to G_{k+1}$ has the following properties.
    \begin{itemize}
        \item There is a subset $A_{k+1} \subset G_{k+1}$ containing the image of $A_k$ (hence the image of $A$) such that  $\Gamma(G_{k+1}, A_{k+1})$ is hyperbolic and the action of $G_{k+1}$ on this space is non-elementary and acylindrical.
        \item $\pi$ is injective on $B_{A_k}(R)$.
        \item $\pi(B_k) \subset \pi(L_k)$, so that the image $H_{k+1}$ of $H_k$ in $G_{k+1}$ is normally generated (as a subgroup of $H_{k+1}$) by $h_{k+1}$.
        \item $\pi(L_k) = \pi(H_k)$ is a $2$-generated, suitable subgroup of $G_{k+1}$.
        \item Every element in $G_{k+1}$ of order $n$ is the image of an element in $G_k$ with order $n$.
        \item $\ker \pi$ is contained in $H_k$.
    \end{itemize}
    We now write $N_{k+1}$ for the kernel of the quotient $G \to G_{k+1}$, and remark that $N_k < N_{k+1} < H$; in fact $N_{k+1}$ is also the kernel of the restriction $H \to H_{k+1}$.

    Since $\pi(L_k)$ is a suitable subgroup of $G_{k+1}$, it follows that $G_{k+1}$ has no non-trivial finite normal subgroup, hence (\ref{enu: induction - acyl}) holds.
    Moreover the normal subgroup $H_{k+1}$ is non-trivial, therefore infinite.
    By construction $H_{k+1}$ is normally generated by $h_{k+1}$.
    However $H_k$ was normally generated by each element in $\{h_1, \dots, h_k\} \setminus N_k$, hence (\ref{enu: induction - normal closure}) holds.
    Points (\ref{enu: induction - inj}) and (\ref{enu: induction - diagram}) directly follow from the construction and the previous discussion.
    This completes the induction step.

    We now define $N$ to be the increasing union of all the subgroups $N_k$ and let $H_\infty \coloneqq H/N$ and $G_\infty \coloneqq G/N$.
    By construction, every non-trivial element $h \in H$ is either in $N$ or normally generates $H_\infty$.
    Hence $H_\infty$ is simple.
    As a quotient of every $H_k$, it is $2$-generated.
    Suppose by contradiction that $H_\infty$ is finitely presentable. 
    Then $N$ is the normal closure of a finite subset of $H$. 
    Since $N$ is the directed union of the $N_k$, we deduce that $N = N_k$, for $k$ large enough. 
    Therefore $H_\infty = H_k$. 
    But this is a contradiction, because $H_\infty$ is simple, while $H_k$ is acylindrically hyperbolic, being an infinite normal subgroup of the acylindrically hyperbolic group $G_k$ (Lemma \ref{lem:normal}).
    It directly follows from the construction that the projection $G \to G_\infty$ is injective on $B_A(R)$.
    
    An induction argument shows that every element of order $n$ in $H_\infty$ has a pre-image $g$ of order $n$ in $G$.
    However the kernel of the projection $G \to G_\infty$ is contained in $H$, hence $g$ belongs to $H$.
    Consequently $\tau(H_\infty) \subset \tau(H)$.
    Let us prove the converse inclusion.
    Let $h \in H$ be a finite order element.
    In particular, $\left<h \right>$ is elliptic, thus it is conjugate to a finite subgroup contained in $B_A(R)$, on which the projection $G \to G_\infty$ is injective (Lemma~\ref{lem:elliptic}).
    Hence $\left< h \right>$ embeds in $G_\infty$, and a posteriori in $H_\infty$. This shows that $\tau(H) \subset \tau(H_\infty)$.
    
    Finally, note that the following diagram commutes
    \begin{center}
         \begin{tikzcd}
         1 \arrow[r] & H  \arrow[r] \arrow[d] & G \arrow[r] \arrow[d] & Q \arrow[r]\arrow[d, "\cong"] & 1 \\
         1 \arrow[r] & H_\infty \arrow[r] & G_\infty \arrow[r] & Q \arrow[r] & 1
        \end{tikzcd}
    \end{center}
    and its rows are short exact sequences.
    The action by conjugation of $G_\infty$ on $H_\infty$ defines a morphism $G_\infty \to \Aut(H_\infty)$, so that the image of any element of $H_\infty$ is inner.
    Hence it induces a morphism $\chi \colon Q \to \Out(H_\infty)$.
    Let us prove that $\chi$ is injective.
    Consider an element $q \in Q$ and $g \in G$ a pre-image of $q$.
    Assume that $\chi(q) = 1$.
    In particular, there is $u \in H$ such that $gu$ centralizes $H_\infty$.
    Since $H$ is finitely generated, this property is ensured by finitely many relations in $N$, therefore there exists $k \in \mathbb N$ such that $gu$ centralizes $H_k$.
    Since $H_k$ is non-elementary (Lemma \ref{lem:normal}), $\left<gu\right>$ must be finite \cite[Proposition 6]{WPD}. But $H_k$ is a suitable subgroup of $G_k$ (Lemma \ref{lem:normal:suitable}) so it cannot normalize a non-trivial finite subgroups. It follows that $gu = 1$ in $G_k$; in particular $g \in H_k$ and so $q = 1$.
\end{proof}

As a corollary we obtain Theorem \ref{intro:thm:general}:

\begin{corollary}
\label{cor:general}
    Let $G$ be an acylindrically hyperbolic group. Let $H < G$ be finitely generated, infinite and normal. Then there exists a subgroup $N < H$, normal in $G$, such that $H/N$ is $2$-generated, infinite, and simple.
\end{corollary}

\begin{proof}
    Let $G_0 \coloneqq G/K(G)$, which is an acylindrically hyperbolic group with no non-trivial finite normal subgroup (Theorem \ref{thm:K}). 
    Let $H_0$ be the image of $H$ in $G_0$. 
    Theorem \ref{thm:general} gives a subgroup $N_0 < H_0$, normal in $G_0$, such that $H_0 / N_0$ is $2$-generated, infinite and simple. 
    Let $N'$ be the pre-image of $N_0$ in $G$, and let $N \coloneqq H \cap N'$. 
    Then $N$ is a subgroup of $H$, normal in $G$, and $H/N \cong H_0 / N_0$ is infinite and simple.
\end{proof}

Combining with Lemma \ref{lem:char}, we obtain Corollary \ref{intro:cor:aut}:

\begin{corollary}
\label{cor:aut}
    Let $\Gamma$ be a finitely generated group such that $[\Gamma : Z(\Gamma)] = \infty$ and $\Aut(\Gamma)$ is acylindrically hyperbolic. 
    Then $\Gamma$ admits a characteristic subgroup $N$ containing $Z(\Gamma)$ such that $\Gamma / N$ is $2$-generated, infinite, and simple.
\end{corollary}

\begin{proof}
    The hypothesis implies that $\Inn(\Gamma)$ is an infinite normal subgroup of the acylindrically hyperbolic group $\Aut(\Gamma)$. Thus the result is a combination of Corollary \ref{cor:general} and Lemma \ref{lem:char}.
\end{proof}

\begin{example}
\label{ex:list}
    Corollary \ref{cor:aut} applies when $\Gamma$ is a finitely generated group such that one of the following holds.
    \begin{itemize}
        \item $\Gamma$ is non-elementary hyperbolic \cite{auto:oneend, auto:infend}.
        \item $\Gamma$ is non-elementary hyperbolic relative to a collection of finitely generated subgroups, none of which is relatively hyperbolic \cite[Theorem 1.3]{auto:infend}.
        \item $\Gamma$ has infinitely many ends \cite[Theorem 1.1]{auto:infend}.
        \item $\Gamma$ is a graph product of finitely generated groups over a finite graph that is not a join and not reduced to a single vertex \cite[Theorem A.27]{escalierhorbez} (see also \cite[Theorem E]{auto:graph1} and \cite[Theorem 1.5]{auto:graph2}).
        \item $\Gamma$ is a torsion-free Artin group, whose underlying graph is connected and has a separating vertex that is not in the center of $\Gamma$ \cite[Theorem E]{auto:artin}.
    \end{itemize}
    This is not the first time that acylindrical hyperbolicity of automorphism groups is used to produce $\Aut$-invariant versions of previously known constructions \cite{osin:aut, ffw}. It would be interesting to see more instances of this.\footnote{More instances did appear since the first version of this paper, see \cite{fff:separable, rinfty}.}
\end{example}

\section{SQ-universality}

Corollary~\ref{cor:aut} proves the existence of a simple characteristic quotient of $F_n$. In this section, we prove that there is a wealth of such quotients, by embedding arbitrary countable groups. We get the following more precise and more general version of Theorem \ref{intro:thm:answer:strong}.

\begin{theorem}
\label{thm:answer}
    Let $\Gamma$ be a torsion-free non-elementary hyperbolic group, and let $C$ be a finite subset of $\Gamma$. 
    Let $L$ be a countable group.
    Then $\Gamma$ admits a characteristic subgroup $N$ with the following properties.
    \begin{itemize}
        \item $\Gamma / N$ is simple.
        \item $\Gamma / N$ is $2$-generated but not finitely presentable.
        \item The quotient $\Gamma \to \Gamma / N$ is injective on $C$.
        \item The group $L$ embeds in $\Gamma / N$, and $\tau(\Gamma / N) = \tau(L)$.
        \item $\Out(\Gamma)$ embeds in $\Out(\Gamma / N)$.
    \end{itemize}
\end{theorem}

As a corollary, we obtain the following generalization of Theorem \ref{intro:thm:answer}.

\begin{corollary}
\label{cor:answer}
    Let $\Gamma$ be a torsion-free non-elementary hyperbolic group. For every set of primes $\mathcal{P}$, there exist continuum many, pairwise non-isomorphic, infinite, $2$-generated, simple characteristic quotients $\Gamma/N$ of $\Gamma$ such that $\tau(\Gamma/N) = \mathcal{P}$. In particular, $\Gamma$ admits continuum many torsion-free, characteristic, simple quotients.
\end{corollary}

\begin{proof}
	Let $E$ be a countable group such that $\tau(E) = \mathcal{P}$, for example $E \coloneqq \ast_{p \in \mathcal{P}} \mathbb{Z}/p\mathbb{Z}$.
    Let $\{ L_i \}_{i \in I}$ be a family of continuum many pairwise non-isomorphic finitely generated torsion-free groups, such as Camm's simple amalgamated products \cite{camm}. 
    Then by Theorem \ref{thm:answer}, for each $i \in I$ there exists a characteristic subgroup $N_i$ of $\Gamma$ such that $\Gamma / N_i$ is $2$-generated, infinite and simple; moreover $E \ast L_i$ embeds in $\Gamma / N_i$ and $\tau(\Gamma / N) = \tau(E \ast L_i) = \tau(E) = \mathcal{P}$.
    
    Now $\Gamma / N_i$ is a countable group containing the finitely generated subgroup $L_i$. 
    Since every countable group contains only countably many finitely generated subgroups, it follows that the family $\Gamma / N_i$ achieves continuum many isomorphism types.
\end{proof}

\begin{remark}
    We could also more directly obtain continuum many pairwise non-isomorphic groups by varying $\mathcal{P}$ among all infinite sets of primes.
\end{remark}

Before proving Theorem~\ref{thm:answer}, we start with the following auxiliary result.

\begin{proposition}
\label{res: quotient with prescribed torsion}
	Let $G$ be a group and $A \subset G$ a generating set of $G$.
	Assume that $\Gamma(G,A)$ is hyperbolic and the action of $G$ on this space is acylindrical and non-elementary.
    Let $H < G$ be a subgroup that is suitable with respect to its action on $\Gamma(G,A)$. 
    Let $L$ be a countable group and $R \in \mathbb R_+$.
    There is a quotient $\pi : G \to \bar G$  such that if we denote $\bar H \coloneqq \pi(H)$, and $N \coloneqq \ker(\pi)$, then the following hold.
    \begin{itemize}
        \item There is a subset $\bar A \subset \bar G$ containing $\pi(A)$ such that  $\Gamma(\bar G, \bar A)$ is hyperbolic and the action of $\bar G$ on this space is non-elementary and acylindrical. Moreover $K(\bar G) = 1$.
        \item $\pi : G \to \bar G$ is injective on $B_A(R)$.
        \item $L$ embeds as a subgroup of $\bar H$, and its image is elliptic for the action on $\Gamma(\bar G, \bar A)$.
        \item If in addition, $H$ is normal in $G$, then $N < H$ and $\tau(\bar H) = \tau(L) \cup \tau(H)$.
    \end{itemize}		        
\end{proposition}

\begin{proof}
	We start by noticing that, since every countable group embeds in a finitely generated group preserving the values of torsion \cite[Theorem~IV.3.1]{Lyndon:2001wm}, we may reduce to the case in which $L$ is finitely generated.
    Let $P$ denote the free product $G \ast L$.
	The set $A' \coloneqq A \cup L$ generates $P$.
	Since $P$ is hyperbolic relative to $\{G,L\}$, it is known that $\Gamma(P,A')$ is $\delta$-hyperbolic (for some $\delta \in \mathbb R_+$) and the action of $P$ on this space is non-elementary and acylindrical \cite[Theorem 2.7]{Abbott:2019aa}.
	On the one hand, $H$ does not normalize any non-trivial finite subgroup of $P$.
	On the other hand, the graph $\Gamma(G,A)$ admits a $G$-equivariant isometric embedding in $\Gamma(P,A')$.
	Hence, the action of $H$ on $\Gamma(P, A')$ is also non-elementary, and therefore the subgroup $H$ is suitable for the action of $P$ on $\Gamma(P,A')$.	

	Without loss of generality we can assume that $R \geq \max\{ 1, 10\delta\}$.
	Fix a finite generating set $B$ of $L$.
	Note that by construction $L$ is contained in the ball of radius $1$ of $\Gamma(P, A')$.
	Now we apply Theorem~\ref{thm:sc} to $P$, and obtain a quotient $\pi : P \to \bar{P}$ with the following properties.
	Denote by $\bar B$, $\bar L$, $\bar H$, and $\bar G$ the respective images of $B$, $L$, $H$ and $G$ under $\pi$.
	\begin{itemize}
		\item There is a subset $\bar A$ of $\bar P$ containing the image of $A'$ such $\Gamma(\bar P, \bar A)$ is hyperbolic.
		The action of $\bar P$ on this space is non-elementary and acylindrical.
		Moreover $K(\bar P) = \{1\}$.
		\item $\bar B$ is contained in $\bar H$, hence $\bar L < \bar H < \bar G$.
		\item $\pi$ is injective on $B_{A'}(R) \subset P$, hence $\pi|_G$ is injective on $B_A(R) \subset G$.
		\item Every element of order $n$ in $\bar{P}$ is the image of an element of order $n$ in $P$.
		\item $\ker(\pi)$ is contained in the normal closure of $L\cup H$, which we denote by $\normal{L\cup H}$.
	\end{itemize}
	Let $N \coloneqq G \cap \ker(\pi)$.
	Since $\bar L < \bar G$, the restriction $\pi : G \to \bar P$ is surjective, hence $\bar P = \bar G = G/N$ is a quotient of $G$.
	Note that the diameter of $L$ in $\Gamma(P,A')$ is $1$.
	Hence $L$ embeds in $\bar P$.
	By construction the map $\Gamma(P, A') \to \Gamma(\bar P, \bar A)$ is $1$-Lispchitz, hence $\bar L$ is elliptic for its action $\Gamma(\bar P, \bar A)$.
	
	Assume now that $H$ is normal in $G$.
	Hence the quotient $P / \normal{ L \cup H}$ is exactly $G/H$.
	This means that $G \cap \normal{L \cup H} = H$.
    However $\ker(\pi)$ is contained in $\normal{L \cup H}$, thus $N$ is subgroup of $H$.
	Since $L$ embeds in $\bar H$, we have $\tau(L) \subset \tau(\bar H)$.
	It is proven as in Theorem~\ref{thm:general} that every finite subgroup of $P$ embeds in $\bar{P}$, hence $\tau(H) \subset \tau(\bar H)$.
	Let us prove now the converse inclusion.
	Consider an element $n \in \tau(\bar H)$.
	Let $\bar h$ be an element of $\bar H$ with order $n$.
	We know that $\bar h$ admits a pre-image $h \in P$ with the same order.
	Since $P$ is a free product, either $h$ is conjugate into $L$, in which case $n \in \tau(L)$, or $h$ is conjugate into $G$. 
	Recall that $H$ is normal in $G$, hence so is $\bar H$ in $\bar G$.
	Up to replacing $\bar h$ by a conjugate, we can assume that $h$ actually belongs to $G$, thus to $HN$.
	However the kernel $N$ of the map $G\to \bar G$ is contained in $H$, hence $h$ belongs to $H$.
	This shows that $n \in \tau(H)$.
\end{proof}

\begin{proof}[Proof of Theorem~\ref{thm:answer}]
	Since $\Gamma$ is a non-elementary hyperbolic group, $\Aut(\Gamma)$ is acylindrically hyperbolic (Example \ref{ex:list}), and since $\Gamma$ is torsion-free, $Z(\Gamma) = \{ 1 \}$, and in particular $\Inn(\Gamma) \cong \Gamma$ is infinite. We claim that $\Aut(\Gamma)$ has no non-trivial finite normal subgroup. Indeed, let $K < \Aut(\Gamma)$ be a finite normal subgroup. Then $\Gamma$ acts on $K$ by conjugacy, and thus there exists a finite index subgroup $\Gamma' < \Gamma$ such that the action of $\Gamma'$ on $K$ fixes every point. Therefore every automorphism in $K$ fixes $\Gamma'$ pointwise. But since $\Gamma$ is torsion-free and hyperbolic, this implies that $K$ is trivial \cite[Lemmas 2.2 and 2.3]{root}.
	
	For simplicity of reference, we now let $H \coloneqq \Gamma$, which we identify with $\Inn(\Gamma)$, $G \coloneqq \Aut(\Gamma)$, and $Q \coloneqq \Out(\Gamma)$.
	Let $L$ be a countable group.
	Let $A$ be generating set of $G$ such that $\Gamma(G,A)$ is hyperbolic and the action of $G$ on this space is acylindrical and non-elementary (Theorem \ref{thm:cayley}).
	Since $H$ is an infinite normal subgroup of $G$, it is non-elementary as well (Lemma \ref{lem:normal}).
	According to Proposition~\ref{res: quotient with prescribed torsion}, there is a normal subgroup $N_0$ of $G$ contained in $H$ such that $\bar G \coloneqq G/N_0$ is acylindrically hyperbolic.
	More precisely there is a subset $\bar A$ generating $\bar G$ and containing the image of $A$ such that $\Gamma(\bar G, \bar A)$ hyperbolic and the action of $\bar G$ on this space is non-elementary and acylindrical.
	In addition,  $K(\bar G) = 1$, $C$ embeds in $\bar H$, $L$ embeds in $\bar H$ as an elliptic subgroup, and $\tau(H/N_0) =  \tau(L)$.
	Let $\bar H \coloneqq H/N_0$.
	Since $K(\bar G) = 1$, the quotient $\bar H$ is an infinite normal subgroup of $\bar G$, so that the short exact sequence
	\begin{equation*}
		1 \to \bar H \to \bar G \to Q \to 1
	\end{equation*}
	satisfies the assumption of Theorem~\ref{thm:general}.
	Consequently there is a normal subgroup $N_1 \subset \bar G$ contained in $\bar H$ such that $\bar H/N_1$ is simple, $2$-generated but not finitely presented, with $\tau(\bar H/N_1) = \tau(H/N_0)$.
	Moreover $C$ and $L$ embed in $\bar H/N_1$ and $Q$ embeds in $\Out(\bar H/N_1)$.
	We write $N$ for the pre-image of $N_1$ in $G$,
	which is a normal subgroup of $G$ contained in $H$.
	We conclude by Lemma \ref{lem:char}.
\end{proof}

\begin{remark}
\label{rem:T}
	Using classical arguments in small cancellation theory, we could prescribe other properties on $\Gamma/N$.
	For instance, in the above proof, if we first replace $H$ by a common quotient of $H$ and an acylindrically hyperbolic group with Property (T) \cite[Corollary 1.6]{hull} we would obtain that $H/N$ has Property (T).
    Similarly, one may try to construct characteristic quotients of hyperbolic groups that are complete, finitely presented and Hopfian \cite[Theorem 10.4]{Lyndon:2001wm}, with a prescribed outer automorphism group \cite{rips:out}, and so on.
    However, it is unclear whether any of these would have interesting connections to independent problems, the same way that Wiegold's question \cite[Question 6.45]{kourovka} is connected to the problem of growth sequences of simple groups.
\end{remark}

\section{Quasi-isometric diversity}
\label{s:QI}

Our goal in this section is to strengthen the uncountability result (Corollary \ref{cor:answer}) to moreover distinguish groups up to quasi-isometry.

\begin{theorem}
\label{thm:QI}
    There exist continuum many, pairwise non-quasi-isometric, $2$-generated, infinite, simple, characteristic quotients of $F_n$.
\end{theorem}

We will prove Theorem \ref{thm:QI} by applying a criterion of Minasyan--Osin--Witzel \cite{diversity}. This requires a small amount on background on the \emph{space of marked groups} $\Gn$ ($n \geq 2$) \cite{grigorchuk}. This is the set of normal subgroups of $F_n$, seen as a subspace of $2^{F_n}$ with the product topology; in particular it is a compact metrizable space. We will often identify a normal subgroup $N < F_n$ with the corresponding marked quotient $F_n \to F_n / N$. We will use some standard terminology from descriptive set theory \cite[I.8]{kechris}: A space is \emph{perfect} if it is closed and has no isolated points; a countable intersection of open sets is called a \emph{$G_\delta$-set}. In our settings, a set is  \emph{comeager} if contains a dense $G_\delta$-set (this relies on the Baire Category Theorem).

Following \cite{diversity}, we say that a subspace of $\Gn$ is \emph{quasi-isometrically diverse} if every comeager subset of it has continuum many quasi-isometry classes.
The proof of Theorem \ref{thm:QI} will use the following criterion for quasi-isometric diversity:

\begin{theorem}[{\cite[Corollary 1.2]{diversity}}]
\label{thm:diversity:criterion}
    Every non-empty perfect subset of $\Gn$ containing a dense subset of finitely presented groups is quasi-isometrically diverse.
\end{theorem}

\begin{remark}
    The statement of \cite[Corollary 1.2]{diversity} is about non-empty perfect subsets of $\mathcal{G}$, which is a directed union of the spaces $\Gn$ for $n \geq 2$. Since $\Gn$ is compact, a perfect subset of $\Gn$ is also a perfect subset of $\mathcal{G}$, thus Theorem \ref{thm:diversity:criterion} is a direct consequence of \cite[Corollary 1.2]{diversity}.
\end{remark}

We will apply Theorem \ref{thm:diversity:criterion} as in the proof of \cite[Corollary 4.10]{diversity}. Namely, we will exhibit a set $\mathcal{H}$ of finitely presented marked groups, and a set $\mathcal{S}$ of simple characteristic marked groups, such that $\mathcal{S}$ is comeager in the closure $\cl(\mathcal{H})$. 
Note that the groups that appear in intermediate steps of our construction (Theorem \ref{thm:general}) are not necessarily finitely presented, unless $G = H$. 

To bypass this difficulty we use Corollary \ref{lem:fp}: There exists a $2$-generated, \emph{finitely presented}, acylindrically hyperbolic quotient $G_0$ of $\Aut(F_n)$ such that the composition $F_n \to \Aut(F_n) \to G_0$ is surjective. 
We fix the marked group $G_0$ for the rest of this section. Let $\mathcal{H} \subset \Gn$ denote the set of marked groups that are finitely presented, acylindrically hyperbolic, and factor through $F_n \to G_0$. 
We further let $\mathcal{S}$ denote the subset of $\cl(\mathcal{H})$ consisting of simple groups whose marking factors through $F_n \to G_0$. 
Since $G_0$ is $2$-generated, so is every element in $\mathcal S$.
We make the following key observation:

\begin{lemma}
\label{lem:factor:char}
    Every quotient of $G_0$ is a characteristic quotient of $F_n$. In particular, every marked group in $\mathcal{S}$ is a simple characteristic quotient of $F_n$.
\end{lemma}

\begin{proof}
    By the choice of $G_0$ (Corollary \ref{lem:fp}), a quotient $F_n \to G_0 \to H$ factors as $F_n \to \Aut(F_n) \to G_0 \to H$. The kernel of this composition is characteristic, by Lemma \ref{lem:char}.
\end{proof}

Finally, we need to establish a topological property of $\mathcal{S}$.

\begin{lemma}
\label{lem:comeager}
    The subset $\mathcal{S} \subset \cl(\mathcal{H})$ is comeager.
\end{lemma}

\begin{proof}
    The set of marked groups that factor through $F_n \to G_0$ is open, because $G_0$ is finitely presented. Moreover, simple groups form a $G_\delta$-set in $\Gn$ \cite[Corollary 4.5]{diversity}. Therefore the set of simple quotients of $G_0$ is a $G_\delta$-set, which implies that $\mathcal{S}$ is a $G_\delta$-subset of $\cl(\mathcal{H})$.

    To show that $\mathcal{S}$ is dense in $\cl(\mathcal{H})$, it suffices to show that each $H \in \mathcal{H}$ is an accumulation point of $\mathcal{S}$. A neighbourhood basis of $H \coloneqq F_n / N_H$ consists of the sets
    \begin{equation*}
        U_W \coloneqq \{ N \in \Gn : N \cap W = N_H \cap W \}; 
    \end{equation*}
    where $W \subset F_n$ is finite. By Theorem \ref{thm:general} applied with $H = G$ (and $Q = 1$), there exists an infinite simple quotient $S$ of $H$ such that the finite set $W N_H \subset F_n / N_H = H$ maps injectively into $S$. Thus $S \in U_W$. Moreover, the construction of Theorem \ref{thm:general} produces a sequence of quotients of $H$, each of which is acylindrically hyperbolic, and obtained from the previous one by an application of Theorem \ref{thm:sc}. Since $H$ is finitely presented, this shows that each of the groups in this sequence belong to $\mathcal{H}$, and therefore $S \in \cl(\mathcal{H})$, and so $S \in \mathcal{S}$.
\end{proof}

We are ready to prove Theorem \ref{thm:QI}.

\begin{proof}[Proof of Theorem \ref{thm:QI}]
    Let $G_0$, $\mathcal{H}$, $\mathcal{S}$ be as above. Recall that $\mathcal{S}$ consists of simple characteristic quotients of $F_n$ by Lemma \ref{lem:factor:char}, and that $\mathcal{S} \subset \cl(\mathcal{H})$ is comeager, in particular it is dense.
    By definition $\mathcal{H}$ is also dense in $\cl(\mathcal{H})$, and moreover $\mathcal{H} \cap \mathcal{S} = \emptyset$, because acylindrically hyperbolic groups cannot be simple. Since an isolated point must belong to every dense subset, it follows that $\cl(\mathcal{H})$ has no isolated points, and therefore is perfect.
    Theorem \ref{thm:diversity:criterion} now applies to give quasi-isometric diversity of $\cl(\mathcal{H})$. Since $\mathcal{S}$ is comeager in $\cl(\mathcal{H})$, it follows from the definition that $\mathcal{S}$ has continuum many quasi-isometry classes. Finite groups form a single quasi-isometry class, so we conclude that there are continuum many quasi-isometry classes of $2$-generated, infinite, simple, characteristic quotients of $F_n$.
\end{proof}

\section{Outlook}

We end with two questions on possible generalizations and applications of our results. The first is about how far the list from Example \ref{ex:list} can be extended:

\begin{question}
\label{q:genevois}
    Let $\Gamma$ be a finitely generated acylindrically hyperbolic group. Does $\Gamma$ admit an infinite simple characteristic quotient?
\end{question}

In \cite[Question 1.1]{auto:oneend}, Genevois asked whether the automorphism group of a finitely generated acylindrically hyperbolic group is always acylindrically hyperbolic. In light of Corollary \ref{intro:cor:aut}, an affirmative answer to Genevois's Question implies an affirmative answer to Question \ref{q:genevois}.

\medskip

The second question asks for a middle ground between Lubotzky's Conjecture (Remark \ref{rem:finite}), which is about \emph{finite} simple characteristic quotients, and our construction, which produces \emph{non-finitely presentable} simple characteristic quotients.

\begin{question}
    Let $n \geq 2$. Does there exist a characteristic subgroup $N < F_n$ such that $F_n / N$ is infinite, simple and finitely presentable?
\end{question}

\footnotesize

\bibliographystyle{amsalpha}
\bibliography{ref}

\newcommand{\etalchar}[1]{$^{#1}$}
\providecommand{\bysame}{\leavevmode\hbox to3em{\hrulefill}\thinspace}
\providecommand{\MR}{\relax\ifhmode\unskip\space\fi MR }
\providecommand{\MRhref}[2]{%
  \href{http://www.ams.org/mathscinet-getitem?mr=#1}{#2}
}
\providecommand{\href}[2]{#2}
\begin{thebibliography}{FFIM{\etalchar{+}}26}

\bibitem[ABO19]{Abbott:2019aa}
C.~Abbott, S.~H. Balasubramanya, and D.~Osin, \emph{Hyperbolic structures on
  groups}, Algebr. Geom. Topol. \textbf{19} (2019), no.~4, 1747--1835.
  \MR{3995018}

\bibitem[BB10]{root}
L.~Bartholdi and O.~Bogopolski, \emph{On abstract commensurators of groups}, J.
  Group Theory \textbf{13} (2010), no.~6, 903--922. \MR{2736164}

\bibitem[BF02]{WPD}
M.~Bestvina and K.~Fujiwara, \emph{Bounded cohomology of subgroups of mapping
  class groups}, Geom. Topol. \textbf{6} (2002), 69--89. \MR{1914565}

\bibitem[BH99]{BH}
M.~R. Bridson and A.~Haefliger, \emph{Metric spaces of non-positive curvature},
  Grundlehren der mathematischen Wissenschaften [Fundamental Principles of
  Mathematical Sciences], vol. 319, Springer-Verlag, Berlin, 1999. \MR{1744486}

\bibitem[Bow08]{Bowditch:2008bj}
B.~H. Bowditch, \emph{Tight geodesics in the curve complex}, Invent. Math.
  \textbf{171} (2008), no.~2, 281--300. \MR{2367021}

\bibitem[Bri22]{bridson}
M.~R. Bridson, \emph{Concise presentations of direct products}, Proc. Amer.
  Math. Soc. \textbf{150} (2022), no.~4, 1361--1368. \MR{4375728}

\bibitem[BW05]{rips:out}
I.~Bumagin and D.~T. Wise, \emph{Every group is an outer automorphism group of
  a finitely generated group}, J. Pure Appl. Algebra \textbf{200} (2005),
  no.~1-2, 137--147. \MR{2142354}

\bibitem[Cam53]{camm}
R.~Camm, \emph{Simple free products}, J. London Math. Soc. \textbf{28} (1953),
  66--76. \MR{52420}

\bibitem[CG05]{limit}
C.~Champetier and V.~Guirardel, \emph{Limit groups as limits of free groups},
  Israel J. Math. \textbf{146} (2005), 1--75. \MR{2151593}

\bibitem[Cha00]{champetier}
C.~Champetier, \emph{L'espace des groupes de type fini}, Topology \textbf{39}
  (2000), no.~4, 657--680. \MR{1760424}

\bibitem[CIOS25]{CIOS}
I.~Chifan, A.~Ioana, D.~Osin, and B.~Sun, \emph{Small cancellation and outer
  automorphisms of {K}azhdan groups acting on hyperbolic spaces}, Algebr. Geom.
  Topol. \textbf{25} (2025), no.~9, 5463--5501. \MR{5005652}

\bibitem[CLT26]{counterexample:baby}
W.~Y. Chen, A.~Lubotzky, and P.~H. Tiep, \emph{Finite simple characteristic
  quotients of the free group of rank 2}, Comment. Math. Helv. \textbf{101}
  (2026), no.~2, 305--344. \MR{5042351}

\bibitem[Cou14]{Coulon:2014fr}
R.~Coulon, \emph{On the geometry of {B}urnside quotients of torsion free
  hyperbolic groups}, Internat. J. Algebra Comput. \textbf{24} (2014), no.~3,
  251--345. \MR{3211906}

\bibitem[Del96]{SQ2}
T.~Delzant, \emph{Sous-groupes distingu\'{e}s et quotients des groupes
  hyperboliques}, Duke Math. J. \textbf{83} (1996), no.~3, 661--682.
  \MR{1390660}

\bibitem[DGO17]{DGO}
F.~Dahmani, V.~Guirardel, and D.~Osin, \emph{Hyperbolically embedded subgroups
  and rotating families in groups acting on hyperbolic spaces}, Mem. Amer.
  Math. Soc. \textbf{245} (2017), no.~1156, v+152. \MR{3589159}

\bibitem[EH]{escalierhorbez}
A.~Escalier and C.~Horbez, \emph{Graph products and measure equivalence:
  classification, rigidity, and quantitative aspects}, arXiv preprint
  arXiv:2401.04635. To appear in Mem. Amer. Math. Soc.

\bibitem[Erf95]{erfanian}
A.~Erfanian, \emph{A problem on growth sequences of groups}, J. Austral. Math.
  Soc. Ser. A \textbf{59} (1995), no.~2, 283--286. \MR{1346636}

\bibitem[EW95]{w:finite6}
A.~Erfanian and J.~Wiegold, \emph{A note on growth sequences of finite simple
  groups}, Bull. Austral. Math. Soc. \textbf{51} (1995), no.~3, 495--499.
  \MR{1331443}

\bibitem[FF25]{fff:separable}
F.~Fournier-Facio, \emph{A note on separability in outer automorphism groups},
  Proc. Edinb. Math. Soc. (2) \textbf{68} (2025), no.~3, 705--711. \MR{4943964}

\bibitem[FFIM{\etalchar{+}}26]{rinfty}
F.~Fournier-Facio, H.~Iveson, A.~Martino, W.~Sgobbi, and P.~Wong,
  \emph{Property {$R_\infty$} for groups with infinitely many ends}, Geom.
  Dedicata \textbf{220} (2026), no.~3, Paper No. 25, 30. \MR{5047846}

\bibitem[FFW23]{ffw}
F.~Fournier-Facio and R.~D. Wade, \emph{{\rm {A}ut}-invariant quasimorphisms on
  groups}, Trans. Amer. Math. Soc. \textbf{376} (2023), no.~10, 7307--7327.
  \MR{4636691}

\bibitem[Gen19]{auto:oneend}
A.~Genevois, \emph{Negative curvature in automorphism groups of one-ended
  hyperbolic groups}, J. Comb. Algebra \textbf{3} (2019), no.~3, 305--329.
  \MR{4011668}

\bibitem[Gen24]{auto:graph2}
\bysame, \emph{Automorphisms of {G}raph {P}roducts of {G}roups and
  {A}cylindrical {H}yperbolicity}, Mem. Amer. Math. Soc. \textbf{301} (2024),
  no.~1509. \MR{4808711}

\bibitem[GG13]{garionglasner}
S.~Garion and Y.~Glasner, \emph{Highly transitive actions of {${\rm
  Out}(F_n)$}}, Groups Geom. Dyn. \textbf{7} (2013), no.~2, 357--376.
  \MR{3054573}

\bibitem[GH21]{auto:infend}
A.~Genevois and C.~Horbez, \emph{Acylindrical hyperbolicity of automorphism
  groups of infinitely ended groups}, J. Topol. \textbf{14} (2021), no.~3,
  963--991. \MR{4503954}

\bibitem[GM19]{auto:graph1}
A.~Genevois and A.~Martin, \emph{Automorphisms of graph products of groups from
  a geometric perspective}, Proc. Lond. Math. Soc. (3) \textbf{119} (2019),
  no.~6, 1745--1779. \MR{4295519}

\bibitem[Gri84]{grigorchuk}
R.~I. Grigorchuk, \emph{Degrees of growth of finitely generated groups and the
  theory of invariant means}, Izv. Akad. Nauk SSSR Ser. Mat. \textbf{48}
  (1984), no.~5, 939--985. \MR{764305}

\bibitem[Gro01]{Gromov:2001us}
M.~Gromov, \emph{{${\rm CAT}(\kappa)$}-spaces: construction and concentration},
  Zap. Nauchn. Sem. S.-Peterburg. Otdel. Mat. Inst. Steklov. (POMI)
  \textbf{280} (2001), no.~Geom. i Topol. 7, 100--140, 299--300. \MR{1879258}

\bibitem[GS09]{counterexample:statistic}
S.~Garion and A.~Shalev, \emph{Commutator maps, measure preservation, and
  {$T$}-systems}, Trans. Amer. Math. Soc. \textbf{361} (2009), no.~9,
  4631--4651. \MR{2506422}

\bibitem[Han]{counterexample:baby2}
L.~Hanany, \emph{The yang-baxter equation and characteristic finite simple
  quotients of the free group of rank $2$}, arXiv preprint arXiv:2510.03210.

\bibitem[Hig51]{higman}
G.~Higman, \emph{A finitely generated infinite simple group}, J. London Math.
  Soc. \textbf{26} (1951), 61--64. \MR{38348}

\bibitem[Hul16]{hull}
M.~Hull, \emph{Small cancellation in acylindrically hyperbolic groups}, Groups
  Geom. Dyn. \textbf{10} (2016), no.~4, 1077--1119. \MR{3605028}

\bibitem[JMS]{auto:artin}
O.~Jones, G.~Mangioni, and G.~Sartori, \emph{{JSJ} splittings for all {A}rtin
  groups}, arXiv preprint arXiv:2506.17157.

\bibitem[Kec95]{kechris}
A.~S. Kechris, \emph{Classical descriptive set theory}, Graduate Texts in
  Mathematics, vol. 156, Springer-Verlag, New York, 1995. \MR{1321597}

\bibitem[KM]{kourovka}
E.~I. Khukhro and V.~D. Mazurov, \emph{Unsolved problems in group theory. {T}he
  {K}ourovka notebook}, arXiv preprint arXiv:1401.0300.

\bibitem[LS01]{Lyndon:2001wm}
R.~C. Lyndon and P.~E. Schupp, \emph{Combinatorial group theory}, Classics in
  Mathematics, Springer-Verlag, Berlin, 2001, Reprint of the 1977 edition.
  \MR{1812024}

\bibitem[Lub11]{lubotzky:question}
A.~Lubotzky, \emph{Dynamics of {${\rm Aut}(F_N)$} actions on group
  presentations and representations}, Geometry, rigidity, and group actions,
  Chicago Lectures in Math., Univ. Chicago Press, Chicago, IL, 2011,
  pp.~609--643. \MR{2807845}

\bibitem[MO15]{MO}
A.~Minasyan and D.~Osin, \emph{Acylindrical hyperbolicity of groups acting on
  trees}, Math. Ann. \textbf{362} (2015), no.~3-4, 1055--1105. \MR{3368093}

\bibitem[MO19]{hull:improved}
\bysame, \emph{Acylindrically hyperbolic groups with exotic properties}, J.
  Algebra \textbf{522} (2019), 218--235. \MR{3896954}

\bibitem[MOW21]{diversity}
A.~Minasyan, D.~Osin, and S.~Witzel, \emph{Quasi-isometric diversity of marked
  groups}, J. Topol. \textbf{14} (2021), no.~2, 488--503. \MR{4286046}

\bibitem[Nie24]{nielsen}
J.~Nielsen, \emph{Die {I}somorphismengruppe der freien {G}ruppen}, Math. Ann.
  \textbf{91} (1924), no.~3-4, 169--209. \MR{1512188}

\bibitem[NN51]{counterexample}
B.~H. Neumann and H.~Neumann, \emph{Zwei {K}lassen charakteristischer
  {U}ntergruppen und ihre {F}aktorgruppen}, Math. Nachr. \textbf{4} (1951),
  106--125. \MR{40297}

\bibitem[Obr93]{obratsov}
V.~N. Obraztsov, \emph{Growth sequences of {$2$}-generator simple groups},
  Proc. Roy. Soc. Edinburgh Sect. A \textbf{123} (1993), no.~5, 839--855.
  \MR{1249690}

\bibitem[Ol'79]{Olshanskii:1979b}
A.~Yu. Ol'shanski\u{\i}, \emph{An infinite simple torsion-free {N}oetherian
  group}, Izv. Akad. Nauk SSSR Ser. Mat. \textbf{43} (1979), no.~6, 1328--1393.
  \MR{567039}

\bibitem[Ol'91]{Olshanskii:1991wv}
\bysame, \emph{Geometry of defining relations in groups}, Mathematics and its
  Applications (Soviet Series), vol.~70, Kluwer Academic Publishers Group,
  Dordrecht, 1991, Translated from the 1989 Russian original by Yu. A.
  Bakhturin. \MR{1191619}

\bibitem[Ol'95]{SQ1}
\bysame, \emph{{${\rm SQ}$}-universality of hyperbolic groups}, Mat. Sb.
  \textbf{186} (1995), no.~8, 119--132. \MR{1357360}

\bibitem[Osi10]{Osin:2010dx}
D.~Osin, \emph{Small cancellations over relatively hyperbolic groups and
  embedding theorems}, Ann. of Math. (2) \textbf{172} (2010), no.~1, 1--39.
  \MR{2680416}

\bibitem[Osi16]{osin}
\bysame, \emph{Acylindrically hyperbolic groups}, Trans. Amer. Math. Soc.
  \textbf{368} (2016), no.~2, 851--888. \MR{3430352}

\bibitem[Osi21]{osin:law}
\bysame, \emph{A topological zero-one law and elementary equivalence of
  finitely generated groups}, Ann. Pure Appl. Logic \textbf{172} (2021), no.~3,
  Paper No. 102915, 36. \MR{4172771}

\bibitem[Osi25]{osin:aut}
\bysame, \emph{{${\rm Out}(F_n)$}-invariant probability measures on the space
  of {$n$}-generated marked groups}, Groups Geom. Dyn. \textbf{19} (2025),
  no.~2, 431--444. \MR{4940646}

\bibitem[OT13]{osinthom}
D.~Osin and A.~Thom, \emph{Normal generation and {$\ell^2$}-{B}etti numbers of
  groups}, Math. Ann. \textbf{355} (2013), no.~4, 1331--1347. \MR{3037017}

\bibitem[Pak01]{pak:question}
I.~Pak, \emph{What do we know about the product replacement algorithm?}, Groups
  and computation, {III} ({C}olumbus, {OH}, 1999), Ohio State Univ. Math. Res.
  Inst. Publ., vol.~8, de Gruyter, Berlin, 2001, pp.~301--347. \MR{1829489}

\bibitem[Wie74]{w:finite1}
J.~Wiegold, \emph{Growth sequences of finite groups}, J. Austral. Math. Soc.
  \textbf{17} (1974), 133--141. \MR{349841}

\bibitem[Wie88]{w:tarski}
\bysame, \emph{Is the direct square of every {$2$}-generator simple group
  {$2$}-generator?}, Publ. Math. Debrecen \textbf{35} (1988), no.~3-4,
  207--209. \MR{1005283}

\bibitem[Wis02]{wise}
D.~T. Wise, \emph{The rank of a direct power of a small-cancellation group},
  Proceedings of the {C}onference on {G}eometric and {C}ombinatorial {G}roup
  {T}heory, {P}art {I} ({H}aifa, 2000), vol.~94, 2002, pp.~215--223.
  \MR{1950879}

\bibitem[WW78]{w:fg1}
J.~Wiegold and J.~S. Wilson, \emph{Growth sequences of finitely generated
  groups}, Arch. Math. (Basel) \textbf{30} (1978), no.~4, 337--343. \MR{503347}

\end{thebibliography}

\vspace{0.5cm}

\normalsize

\noindent{\textsc{Université Bourgogne Europe, CNRS, IMB UMR 5584, 21000 Dijon, France}} \\
\noindent{\textit{E-mail address:} \texttt{remi.coulon@cnrs.fr}} \\

\noindent{\textsc{Department of Pure Mathematics and Mathematical Statistics, University of Cambridge, UK}}\\
\noindent{\textit{E-mail address:} \texttt{ff373@cam.ac.uk}}

\end{document}